\documentclass[sn-mathphys,Numbered]{sn-jnl}
\usepackage{graphicx}%
\usepackage{multirow}%
\usepackage{amsmath,amssymb,amsfonts}%
\usepackage{amsthm}%
\usepackage{mathrsfs}%
\usepackage[title]{appendix}%
\usepackage{xcolor}%
\usepackage{textcomp}%
\usepackage{manyfoot}%
\usepackage{booktabs}%
\usepackage{algorithm}%
\usepackage{algorithmicx}%
\usepackage{algpseudocode}%
\usepackage{listings}%
\usepackage{mathtools}
\usepackage{pgfplots}
\newtheorem{theorem}{Theorem}[section]
\newtheorem{lemma}[theorem]{Lemma}
\newtheorem{corollary}[theorem]{Corollary}

\theoremstyle{definition}
\newtheorem{definition}[theorem]{Definition}

\theoremstyle{remark}
\newtheorem{remark}[theorem]{Remark}

\begin{document}

\title[Rigidity results of compact regular $(\kappa,\mu)$-manifolds]{On Certain Rigidity results of compact regular $(\kappa,\mu)$-manifolds}

\author[1]{\fnm{Sannidhi} \sur{Alape}}\email{sannidhi.a@gmail.com}
\equalcont{These authors contributed equally to this work.}

\author[2]{\fnm{Atreyee} \sur{Bhattacharya}}\email{atreyee@iiserb.ac.in}
\equalcont{These authors contributed equally to this work.}

\author*[3]{\fnm{Dheeraj} \sur{Kulkarni}}\email{dheeraj@iiserb.ac.in}
\equalcont{These authors contributed equally to this work.}

\affil[1]{\orgdiv{Department of Mathematics}, \orgname{Indian Institute of Science Education and Research}, \orgaddress{\city{Bhopal}, \postcode{462066}, \state{Madhya Pradesh}, \country{India}}}

\affil[2]{\orgdiv{Department of Mathematics}, \orgname{Indian Institute of Science Education and Research}, \orgaddress{\city{Bhopal}, \postcode{462066}, \state{Madhya Pradesh}, \country{India}}}

\affil*[3]{\orgdiv{Department of Mathematics}, \orgname{Indian Institute of Science Education and Research}, \orgaddress{\city{Bhopal}, \postcode{462066}, \state{Madhya Pradesh}, \country{India}}}

\abstract{In this article, we investigate the Riemannian and semi-Riemannian metrics on the base space of the 
    Boothby-Wang fibration of a closed regular non-Sasakian $(\kappa,\mu)$-manifold. To this end, we study a natural class of deviations of the projection map from being (semi-)Riemannian submersions. We 
    consider deviations that preserve the canonical bi-Legendrian structure on the given $(\kappa,\mu)$-manifold. This approach gives a unified framework to analyze rigidity results in both categories. As a consequence, in the Riemannian category, we obtain uniqueness of Sasakian structure on the given $(\kappa,\mu)$-manifold which orthogonalizes the canonical bi-Legendrian structure. In the semi-Riemannian category, 
    we obtain an explicit description of the finitely many para-Sasakian structures which orthogonalize the canonical bi-Legendrian structure.}

\keywords{$(\kappa,\mu)$-spaces, Regular contact manifold, Bi-Legendrian structures}

\pacs[2020 MSC Classification]{53C15, 53C12, 53C25, 53D10}

\maketitle

\numberwithin{equation}{section}
\section{Introduction}
Riemannian geometry of contact and symplectic manifolds has been an active area of research linking various branches of mathematics and mathematical physics. A particular class of contact metric manifolds, called Sasakian manifolds, introduced by Shigeo Sasaki in \cite{SSM}, has garnered interest from mathematicians and physicists alike, due to its remarkable properties. For instance, it is known that every regular compact Sasakian manifold is naturally ``sandwiched" between two Kähler manifolds. Compact Sasakian manifolds are also a rich source of examples of Einstein manifolds. A generalisation of Sasakian manifolds was defined in \cite{KUM}, where the authors studied contact metric manifolds $(M,\eta, \xi, \phi, g)$ satisfying the $(\kappa,\mu)$-nullity condition
\begin{equation}\label{eq:ku}
    R(X,Y)\xi = \kappa(\eta(Y)X-\eta(X)Y) + \mu(\eta(Y)hX-\eta(X)hY),
\end{equation}
where $\kappa,\mu$ are constants and 2$h$ is the Lie derivative of the tensor field $\phi$ in the direction of the Reeb vector field $\xi$. In addition to being natural generalizations of Sasakian manifolds, $(\kappa, \mu)$-manifolds also generalize the class of manifolds satisfying the condition $R(X,Y)\xi = 0$. They are also invariant under $D_a$-homothetic transformations (see section \ref{subsec:ku} for details). Furthermore, $(\kappa, \mu)$-manifolds are examples of $H$-contact manifolds (\cite{HCM}) and locally $\phi$-symmetric spaces (\cite{LPSCM}). In the case of non-Sasakian $(\kappa,\mu)$-manifolds, the tensor field $h$ has two eigenvalues, the eigenspaces of which define Legendrian distributions on the underlying space. Thus, non-Sasakian $(\kappa,\mu)$-manifolds are naturally bi-Legendrian manifolds.

One of the fundamental results relating regular contact manifolds and symplectic manifolds is due to Boothby and Wang (\cite{BWF}), wherein compact regular contact manifolds are realised as principal $\mathbb{S}^{1}$-bundles over symplectic manifolds with integral symplectic form and vice versa. Metric versions of the Boothby-Wang fibration have been successfully constructed for specific classes of contact metric manifolds. The aforementioned correspondence between Sasakian manifolds and Kähler manifolds is one instance of such a construction (\cite{SKC}), where the projection map is a Riemannian submersion. However, in the case of $(\kappa,\mu)$-manifolds, the Reeb vector field not necessarily being a Killing vector field presents an obstruction for the definition of a natural metric on the base space. 

In Section \ref{sec:riem} of this article, we investigate symplectic metric structures on the manifold which is the base space for the Boothby-Wang fibration whose total space is a closed regular non-Sasakian $(\kappa,\mu)$-manifold. Since the Reeb vector field is not a Killing vector field, we notice that we cannot have a Riemannian submersion. Hence, we consider the class of conformal submersions. But we notice rigidity even in this case (Theorem \ref{thm:confrig}). In fact, we prove the following result which proves that conformal submersions are possible only when the total space is a $K$-contact manifold.
\begin{theorem}\label{thm:confrig}
    Let $\pi : M \rightarrow B$ be the Boothby-Wang fibration from a regular closed contact metric manifold $(M,\eta, \Tilde{g}, \phi)$ to a symplectic metric manifold $(B,\omega, g, J)$. Suppose $\pi$ is a conformal submersion of the underlying Riemannian manifolds. Then, $\pi$ is a Riemannian submersion and $M$ is a $K$-contact manifold.
\end{theorem}
\noindent In light of the above, we consider the tensor field which quantifies the extent to which a metric on the base space fails to define a conformal submersion. We define the error tensor $T$ corresponding to a function $f$ on the total space, using the equation
\begin{equation}\label{eq:main}
    \Tilde{g}(TX, Y) = \Tilde{g}(X,Y)-e^{2f}g(d\pi X, d\pi Y),
\end{equation}
\noindent for all $X,Y \in \Gamma(Ker(\eta))$. We focus on the error tensors for which the Legendrian distributions are preserved. We prove that this condition forces the absolute value of the index of the $(\kappa,\mu)$-manifold to be strictly greater than one. To our surprise, the solutions obtained under these assumptions point to a unique Riemannian metric on the base space. The conclusions are summarized as follows.
\begin{theorem}\label{thm:main}
    Suppose $(T,f,g)$ is a solution to equation \eqref{eq:main} and $hT = Th$. Then,
    \begin{itemize}
        \item The Boeckx index $I_{M}$ satisfies the condition $\vert I_{M} \vert > 1$.
        \item The error tensor $T$ has two eigenvalues, with the eigenspaces coinciding with that of $h$. The eigenvalues, and consequently $T$ itself, are determined by the function $f$.
        \item The metric $g$ is independent of the choice of $f$ and $T$.
    \end{itemize}
\end{theorem}
\noindent The unique solution thus obtained also happens to be a Kähler metric, which has been studied in a different context (\cite{SPS}). We also establish an equivalent formulation of the condition $hT=Th$ in terms of certain geometric properties of the Kähler metric described above. As a consequence, we obtain rigidity of $K$-contact structures which orthogonalize the bi-Legendrian structure associated with the given $(\kappa,\mu)$-manifold. The result is as follows:
\begin{theorem}\label{thm:rig}
    Let $(M,\eta, g,\phi)$ be a regular closed non-Sasakian $(\kappa,\mu)$-manifold with $\vert I_{M} \vert > 1$. There is a unique $K$-contact structure $(M,\eta, \Bar{g},\Bar{\phi})$ such that the bi-Legendrian structure determined by the eigenspaces of $h=\frac{1}{2}\mathcal{L}_{\xi}\phi$ are orthogonal. This unique structure coincides with the canonical Sasakian structure associated with the $(\kappa, \mu)$-manifold $(M,\eta, g, \phi)$.
\end{theorem}

In the final section, we investigate semi-Riemannian metrics associated to the symplectic structure on the base space, motivated by the results in \cite{SPS}. A simple observation rules out the possibility of a non-trivial semi-Riemannian submersion. We thus consider the error tensor which measures the deviation of the projection map from being a semi-Riemannian submersion. In this case, the error equation is given by
\begin{equation}\label{eq:main3}
    \Tilde{g}(T\tilde{X}, \tilde{Y}) = \Tilde{g}(\tilde{X},\tilde{Y})-g(X, Y),
\end{equation}
where $X,Y\in \Gamma(B)$.  The consequences of the error tensor preserving the Legendrian distributions are summarized in the following result.
\begin{theorem}\label{thm:main2}
    Suppose $(T,g)$ is a solution to equation \eqref{eq:main3} and $hT = Th$. Then,
    \begin{itemize}
        \item The Boeckx index $I_{M}$ satisfies the condition $\vert I_{M} \vert < 1$.
        \item The solution set has cardinality $2^{n}$.
    \end{itemize}
\end{theorem}
We also obtain an explicit description of the finitely many semi-Riemannian metrics which are compatible with the symplectic structure on the base space under the given constraints. We recover the para-Kähler structure obtained by the authors of \cite{SPS} in Theorem $3.1 (ii)$ as a particular solution of the error equation.

The paper is organized as follows. Section \ref{sec:prel} is devoted to the preliminaries on the Riemannian Geometry of Contact and Symplectic manifolds, and the Boothby-Wang fibration. In Section \ref{sec:riem}, we study the Riemannian metrics compatible with the base space of a non-Sasakian $(\kappa,\mu)$-manifold and prove the rigidity of the solutions under certain geometric conditions. In Section \ref{sec:semi-riem}, we study the semi-Riemannian metrics compatible with the base space and provide an explicit description of the finite solution space under similar geometric conditions.

\numberwithin{theorem}{subsection}
\numberwithin{equation}{section}
\section{Preliminaries}\label{sec:prel}
\subsection{Contact and Symplectic Manifolds}\label{subsec:consymp}
\begin{definition}\noindent\label{def:consymp}
    \begin{enumerate}
        \item A \textbf{contact form} on an odd-dimensional manifold $M$ is a $1$-form $\eta$ which satisfies the condition
        \begin{equation}\label{eq:cc}
            \eta \wedge (d\eta)^{n} \neq 0,
        \end{equation}
    where dim$(M)=2n+1$. A contact manifold is a pair $(M,Ker(\eta))$ comprising an odd-dimensional manifold $M$ and a hyperplane field given by the kernel of a contact form $\eta$ on $M$. The hyperplane field given by $Ker(\eta)$ is called the contact distribution.
        \item A \textbf{symplectic form} on an even-dimensional manifold $B$ is a closed 2-form $\omega$ which satisfies the condition
        \begin{equation}\label{eq:sc}
            \omega^{n}\neq 0,
        \end{equation}
        where dim$(B) = 2n$. A symplectic manifold is a pair $(B,\omega )$ comprising an even-dimensional manifold $B$ and a symplectic form $\omega$ on it.
    \end{enumerate}
\end{definition}
\begin{remark}\label{rem:multicon}
    There could be multiple contact forms defining the same contact distribution, each of which will be called a \textbf{representative contact form} for the contact distribution. If $\eta_{1}$ and $\eta_{2}$ are two contact forms such that $Ker(\eta_{1})=Ker(\eta_{2})$, then $\eta_{1} = f\eta_{2}$, where $f$ is a smooth non-vanishing function on the manifold. Since we intend to work with specific contact forms, we will use $(M,\eta)$ to denote a contact manifold.
\end{remark}
\begin{definition}\label{def:Reeb}
    Given any contact manifold $(M,\eta)$, there is a unique vector field $\xi$ satisfying the conditions:
    \begin{itemize}
        \item $\eta (\xi) = 1$,
        \item $d\eta (\xi, X) =0$, for any vector field $X$.
    \end{itemize}
    This vector field is called the \textbf{Reeb vector field} associated to the contact manifold $(M,\eta)$. Throughout the article, $\xi$ will be used to denote the Reeb vector field of the contact manifold under discussion.
\end{definition}
\noindent For more details about contact and symplectic manifolds, we refer the reader to \cite{GCT}.

We restrict our attention to a class of contact manifolds for which the Reeb vector field is regular (a non-vanishing vector field for which there is a flow box around each point which is pierced at most once by any integral curve). Such contact manifolds are called \textbf{regular contact manifolds}. We have the following well-known result in \cite{BWF} which describes a correspondence between closed regular contact manifolds and symplectic manifolds with integral symplectic form.
\begin{theorem}[\cite{BWF}]\label{thm:bw}
    Let $(M,\eta ')$ be a closed regular contact manifold. There is a representative contact form $\eta$ for the contact structure which generates a free $\mathbb{S}^1$-action on $M$. The Reeb vector field $\xi$ associated with this contact form has orbits isomorphic to $\mathbb{S}^1$ which are the fibers of a principal $\mathbb{S}^1$-bundle $\pi : M \rightarrow B$ for a symplectic manifold $(B, \omega)$ such that $\eta$ is a connection form with curvature form $\omega$, i.e., $\pi^{*}\omega = d\eta$.
    
    Conversely, if $(B, \omega)$ is a symplectic manifold such that the class $[\omega]$ is integral, then the principal $\mathbb{S}^1$-bundle $M$ associated to $[\omega ]$ is a regular contact manifold with contact form $\eta$ such that  $d\eta = \pi^{*}\omega$.
\end{theorem}
Throughout the article, $\pi$ will denote the projection map of the Boothby-Wang fibration. Since the $\mathbb{S}^{1}$-action on $M$ which makes $M$ a principal circle bundle over $B$ is given by the flow $\theta$ of the Reeb vector field $\xi$ of $\eta$, we have, for all $t\in [0,1]$, $p\in M$, $\pi(\theta_{t}(p)) = \pi(p)$, where $t$ parametrizes $\mathbb{S}^{1}$. Taking the derivative of the above equation, we get, $d\pi_{\theta_{t}(p)}({d\theta_{t}}_{p}) = d\pi_{p}$. Thus, the pushforward of vectors in the direction of the Reeb vector field yields vectors which project onto the same vector on the base space $B$.
\begin{definition}\label{def:hvf}
    For every vector field $X$ in $B$, the lift obtained by choosing vectors in the contact distribution is a $\xi$-invariant vector field, which will be referred to as a \textbf{horizontally lifted vector field}.
\end{definition}
Throughout this article, we use horizontally lifted vector fields for computations. The horizontal lift of a vector field $X$ in $B$ will be denoted by $\tilde{X}$. By the observation above, we have $\mathcal{L}_{\xi}{\tilde{X}} = 0$.
\subsection{Contact Metric Manifolds}\label{subsec:CMM}
\begin{definition}\label{def:metconsymp}\noindent
    \begin{enumerate}
        \item A \textbf{contact metric manifold} is a 4-tuple $(M,\eta , g,\phi)$, where $(M,\eta)$ is a contact manifold (with Reeb vector field $\xi$), $g$ is a Riemannian metric on $M$, and $\phi$ is a (1,1)-tensor field, which satisfies:
        \begin{itemize}
            \item $\eta = i_{\xi}g$,
            \item $\phi^{2} = -I + \eta \otimes \xi$, and $d\eta (X,Y) =g(X,\phi Y)$,
        \end{itemize}
        for all $X,Y \in \Gamma(M)$. In this case, $g$ is called an associated metric to the contact form $\eta$.
        \item A \textbf{symplectic metric manifold} is a 4-tuple $(B, \omega, g, J)$, where $(B,\omega)$ is a symplectic manifold, $g$ is a Riemannian metric on $B$ and $J$ is an almost complex structure on $B$, which satisfies 
        \begin{equation}\label{eq:sympcomp}
            g(X,JY) = \omega(X,Y),
        \end{equation}
        for all $X,Y \in \Gamma(M)$.
    \end{enumerate}
\end{definition}
\noindent Given a contact manifold (resp. a symplectic manifold), there exists a Riemannian metric $g$ associated with the contact form (resp. symplectic form).

Consider a closed symplectic manifold $(B,\omega)$ such that $\omega$ is an integral symplectic form. Let $g$ be an associalted metric on $B$ along with an almost complex structure $J$ which makes $(B,\omega,g,J)$ a symplectic metric manifold. Let $(M,\eta)$ be the contact manifold which is the total space of the Boothby-Wang fibration over $(B, \omega)$. We can define a Riemannian metric $\tilde{g}$ and a (1,1)-tensor field $\phi$ on $(M, \eta)$ as specified below, which makes $(M,\eta, \phi, g)$ a (regular) contact metric manifold.
\begin{equation*}
    \phi X = \widetilde{Jd\pi X}\text{ and }\tilde{g}=\pi^{*}g + \eta \otimes \eta .
\end{equation*}
Note that the horizontally lifted vector field also coincides with the horizontal lift with the respect to the metric $\tilde{g}$, defined on $M$. The procedure described above yields a contact metric manifold for which the Reeb vector field is a Killing vector field, i.e., the flow of the Reeb vector field is by isometries.
\begin{definition}
    A contact metric manifold $(M,\eta, g, \phi)$ for which the Reeb vector field is a Killing vector field, i.e., $\mathcal{L}_{\xi}g=0$, is called a $\boldsymbol{K}$\textbf{-contact} manifold.
\end{definition}
\begin{remark}
    One can verify that $(M,\eta, g,\phi)$ is a $K$-contact manifold if and only if the $(1,1)$-tensor field $h=\frac{1}{2}\mathcal{L}_{\xi}\phi$ identically vanishes. The tensor field $h$ plays an important role in the study of contact metric manifolds which are not $K$-contact.
\end{remark}
\noindent The construction described above gives a Boothby-Wang correspondence between regular closed $K$-contact manifolds and closed symplectic metric manifolds with integral symplectic form. In this case, $\pi$ is a Riemannian submersion, i.e., the differential map $d\pi$ is an isometry when restricted to the contact distribution.

There is a natural almost complex structure on the symplectizaton of a contact metric manifold. Requiring the almost complex structure thus obtained to be integrable (making the manifold Kähler) gives rise to an interesting class of contact metric manifolds, introduced by Shigeo Sasaki, in \cite{SSM}.
\begin{definition}
    Given any contact metric manifold $(M, \eta , g,\phi)$, one can define an almost complex structure $J$ on the symplectization $M\times \mathbb{R}$ given by
    \begin{equation}
        J\left( X, f\partial_{t}\right)= \left( \phi X -f\xi, \eta (X)\partial_{t} \right),
    \end{equation}
    where $X\in \Gamma(M)$, and $f\in C^{\infty}(M\times \mathbb{R})$. If $J$ is integrable, then $(M,\eta,g,\phi)$ is called a \textbf{Sasakian manifold}.
\end{definition}
The following is a result which facilitates the reformulation of the condition specified in the definition above, in terms of the structure tensors of the contact metric manifold.
\begin{theorem}
    For a contact metric manifold $(M,\eta , g,\phi)$, the following are equivalent:
    \begin{itemize}
        \item The almost complex structure $J$ defined above, on the symplectization is integrable.
        \item $(\nabla_X\phi)Y = g(X,Y)\xi -\eta(Y)X$, for all vector fields $X,Y$.
        \item $R(X,Y)\xi = \eta(Y)X - \eta(X)Y$, for all vector fields $X,Y$.
    \end{itemize}
\end{theorem}
\begin{remark}
    Sasakian manifolds are $K$-contact but the converse only holds in dimension 3.
\end{remark}
For a more details about Riemannian geometry of contact and symplectic manifolds, see \cite{BRCS}.
The following is a result asserting the existence of a Boothby-Wang correspondence between Sasakian manifolds and Hodge manifolds (Kähler manifolds with integral symplectic form).
\begin{theorem}[\cite{SKC}]
A Boothby-Wang bundle over a Kähler manifold with integral symplectic form is Sasakian. Conversely, every regular compact Sasakian manifold is a Boothby-Wang bundle over a Kähler manifold.
\end{theorem}
We refer the reader to \cite{SG}, for a detailed treatment of the correspondence between Sasakian and Kähler manifolds. 
\subsection{\texorpdfstring{Contact metric $\boldsymbol{(\kappa,\mu)}$-manifolds}{}}\hfill\label{subsec:ku}

A generalization of Sasaskian manifolds was given in \cite{KUM} where the authors define and study a class of contact metric manifolds satisfying the nullity condition
\begin{equation}
    R(X,Y)\xi = \kappa(\eta(Y)X - \eta(X)Y) + \mu (\eta(Y)hX - \eta(X) hY),
\end{equation}
for some real numbers $\kappa,\mu$. Contact metric manifolds satisfying the above nullity condition are called $(\kappa,\mu)$-manifolds. The authors proved that, in dimensions 5 and higher, the curvature of these manifolds is completely determined by the values $\kappa$ and $\mu$. They are also invariant under $\mathcal{D}_{a}$-homothetic transformations. By a $\mathcal{D}_{a}$-homothetic transformation of non-zero constant $a$, of a contact metric manifold $(M, \eta, g, \phi)$, we mean a change of the structure tensors given by
\begin{equation*}
    \tilde{\eta} = a \eta,\quad \tilde{g}= ag+a(a-1)\eta \otimes \eta, \quad \tilde{\phi} = \phi.
\end{equation*}
One can easily check that the above changes give rise to a contact metric manifold $(M,\Tilde{\eta},\tilde{g},\Tilde{\phi})$. Note that the Reeb vector field $\Tilde{\xi}$ and the tensor field $\Tilde{h}$ of the manifold thus obtained are given by
\begin{equation*}
    \Tilde{\xi}=\frac{1}{a}\xi, \quad \Tilde{h}=\frac{1}{a}h.
\end{equation*}
If one starts out with a contact metric manifold whose Reeb vector field satisfies the $(\kappa,\mu)$-nullity condition \eqref{eq:ku}, then the Reeb vector field of the contact metric manifold obtained by the $\mathcal{D}_{a}$-homothetic transformation by non-zero constant $a$ satisfies the nullity condition
\begin{equation*}
    \tilde{R}(X,Y)\Tilde{\xi} = \left(\frac{\kappa + a^{2}-1}{a^{2}}\right)\left(\Tilde{\eta}(Y)X - \eta(X)Y\right)+\left(\frac{\mu+2a-2}{a}\right)\left(\Tilde{\eta}(Y)hX-\tilde{\eta}(X)hY\right).
\end{equation*}
In other words, the nullity condition \ref{eq:ku} is preserved, for constants $\tilde{\kappa}=\frac{\kappa + a^{2}-1}{a^{2}}$ and $\tilde{\mu}= \frac{\mu+2a-2}{a}$.

A local classification of non-Sasakian $(\kappa,\mu)$-manifolds was given by Boeckx, in \cite{BI}, in terms of the index
\begin{equation*}
    I_{M}=\frac{1-\tfrac{\mu}{2}}{\sqrt{1-\kappa}}.
\end{equation*}
It was shown that $I_{M}$ is invariant under $\mathcal{D}_{a}$-homothetic transformations. Furthermore, any two $(\kappa,\mu)$-manifolds have the same index if and only if they are locally isometric upto $\mathcal{D}_{a}$-homothetic transformations.

The $(1,1)$-tensor field $h$ is symmetric with respect to the metric and has many interesting properties in the case of $(\kappa,\mu)$-manifolds. The eigenspaces of $h$ play an important role in the study of $(\kappa,\mu)$-manifolds as is evident from the following result in \cite{KUM}.
\begin{theorem}[\cite{KUM}]
    Let $(M,\eta , g, \phi)$ be a $(\kappa, \mu)$-manifold. Then,
    \begin{equation*}
        h^{2}=-(1-\kappa)\phi^{2}.
    \end{equation*}
    Thus, $\kappa \leq 1$. Moreover, $\kappa = 1$ if and only if $M$ is Sasakian. If $M$ is non-Sasakian, then $h$ has three eigenvalues 0, $\lambda$ and $-\lambda$, where $\lambda = \sqrt{1-\kappa}$. The tangent bundle of $M$ admits a decomposition into three mutually orthogonal and integrable distributions $\mathscr{D}_{h}(\lambda),\, \mathscr{D}_{h}(-\lambda)$ and $\mathscr{D}_{h}(0)$ given by the eigenspaces of $h$ corresponding to the eigenvalues $\lambda,\, -\lambda$ and 0, respectively. 
\end{theorem}
To conclude this subsection, we state some formulae involving the structure tensors of a contact metric manifold.
\begin{theorem}\label{thm:ghprop}
    Let $(M,\eta, g,\phi)$ be a contact metric manifold. Then,
    \begin{itemize}
        \item The tensor field $h$ anti-commutes with $\phi$, i.e., $h\phi +\phi h =0$. Thus, $\phi v$ is an eigenvector of $h$ for every eigenvector $v$ of $h$,
        \item $\mathcal{L}_{\xi}g(X,Y) = 2g(hX,\phi Y)$, for vector fields $X,Y$ in $M$,
        \item For a non-Sasakian $(\kappa,\mu)$-manifold, $\mathcal{L}_{\xi}h = (2-\mu)\phi h + 2(1-\kappa)\phi$.
    \end{itemize}
\end{theorem}
\subsection{Bi-Legendrian manifolds}\hfill\label{subsec:bil}
\begin{definition}\label{def:leg}
    Let $(M^{2n+1},\eta)$ be a contact manifold. A \textbf{Legendrian distribution} on $M$ is an $n$-dimensional distribution $L$ of $Ker(\eta)$ such that $d\eta(X,Y)=0$ for all $X,Y \in \Gamma (L)$. Further, if $L$ is integrable, then it defines a \textbf{Legendrian foliation}.
\end{definition}
A contact manifold equipped with a pair of complementary Legendrian distributions is called a bi-Legendrian manifold. If $(M^{2n+1},\eta,g,\phi)$ is a non-Sasakian $(\kappa,\mu)$-manifold, then the subspaces $\mathscr{D}_{h}(\lambda)$ and $\mathscr{D}_{h}(-\lambda)$ define orthogonal Legendrian distributions of $Ker(\eta)$. Thus, any non-Sasakian $(\kappa, \mu)$-manifold is naturally endowed with a pair of complementary integrable Legendrian distributions of dimension $n$ each, making it a bi-Legendrian manifold. Legendrian manifolds have been studied by Pang, Liberman and Jayne (see \cite{LF}, \cite{LFC} and \cite{LFCM}). This was used in \cite{BLKU} to study the interplay between $(\kappa,\mu)$-manifolds and bi-Legendrian manifolds. The authors proved the following result which gives a characterization of $(\kappa,\mu)$-manifolds in terms of the bi-Legendrian structure.
\begin{theorem}\cite{BLKU}
    A contact metric manifold $(M,\eta, g,\phi)$ is a $(\kappa,\mu)$-manifold if and only if it admits an orthogonal bi-Legendrian structure $(L_{1}, L_{2})$ such that the corresponding biligendrian connection $\bar{\nabla}$ satisfies $\Bar{\nabla}\phi=0$ and $\Bar{\nabla}h= 0$. The biligendrian structure $(L_{1},L_{2})$ coincides with the eigenspaces of $h$.
\end{theorem}
Thus, it is natural to look at the bi-Legendrian structure on the given $(\kappa,\mu)-$manifold and investigate error tensors for which the decomposition is preserved. More generally, throughout the article we will impose conditions which preserve the bi-Legendrian structure of the contact manifold.

\subsection{Semi-Riemannian metrics and Symplectic Geometry}\hfill\label{subsec:semir}

Semi-Riemannian metrics are generalizations of Riemannian metrics obtained by relaxing the condition of positive-definiteness. Instead, we only require that the product is non-degenerate. Compatibility between a semi-Riemannian metric and a symplectic manifold is given as follows.
\begin{definition}\label{def:parak} 
    An \textbf{almost para-Kähler} manifold is a $4$-tuple $(B,\omega , g, F)$ comprising a symplectic manifold $(B,\omega)$, a semi-Riemannian metric $g$ and a $(1,1)$-tensor field $F$ such that $F^{2} = I$ and $\omega(X,Y) = g(X,FY)$. If $\nabla F = 0$, where $\nabla$ is the Levi-Civita connection of $g$, then $(B, \omega,g, F)$ is said to be a \textbf{para-Kähler} manifold.
\end{definition}
\noindent For a detailed account of para-Kähler geometry, we refer the reader to \cite{SPG}.

Since we are interested in geometric properties of submersions, we recall the definition of a semi-Riemannian submersion here, which was introduced by O'Neill in \cite{SRG}.
\begin{definition}
    A submersion between semi-Riemannian manifolds $(M,\tilde{g})$ and $(B,g)$ is said to be a \textbf{semi-Riemannian submersion} if 
    \begin{itemize}
        \item The fibers $\{\pi^{-1}(b)\}_{b\in B}$ are semi-Riemannian submanifolds of $M$, and,
        \item The derivative map $d\pi$ preserves the scalar product of vectors normal to the fibers.
    \end{itemize}
\end{definition}

\numberwithin{theorem}{section}
\section{Riemannian metrics on the base space}\label{sec:riem}
In this section, we investigate the properties of a compatible metric on the base space of the Boothby-Wang fibration whose total space is a closed regular non-Sasakian $(\kappa , \mu)$-contact manifold. The projection map of the Boothby-Wang fibration $\pi$ cannot be a Riemannian submersion. This is because, a Riemannian submersion would force $\mathcal{L}_{\xi}g = 0 $, which in turn forces the $(\kappa, \mu)$-contact manifold to be $K$-contact and hence, Sasakian. Next, we show that the projection map $\pi$ being a conformal submersion (a submersion for which the differential map preserves angles when restricted to the contact distribution) also forces the $(\kappa,\mu)$-structure to be Sasakian. To this end, we prove Theorem \ref{thm:confrig}.
\begin{proof}[Proof of Theorem \ref{thm:confrig}]
    Let $p\in M$. We can construct an orthonormal (with respect to $\Tilde{g}_{p})$ basis of $Ker(\eta_{p})$, of the form $\{e_{1}, e_{2}, \dots , e_{n}, \phi e_{1}, \phi e_{2},\dots , \phi e_{n}\}$ as follows. Choose any unit vector $e_{1}$. Since $\Tilde{g}(v,w) = \Tilde{g}(\phi v, \phi w)$, $\phi e_{1}$ is a unit vector as well. Since $d\eta(v,w) = \Tilde{g}(v, \phi w)$, $\phi e_{1}$ is orthogonal to $e_{1}$. Considering the subspace orthogonal to the span of $\{e_{1}, \phi e_{1}\}$ in $Ker(\eta_{p})$, we can continue the procedure described above and obtain the necessary basis in $n$ steps.
    Suppose $\pi$ is a conformal submersion. There exists a function $f$ such that
    \begin{equation*}
        \Tilde{g}(V,W) = e^{2f}g(d\pi V, d\pi W), 
    \end{equation*}
    for all vector fields $V,W \in \Gamma(Ker(\eta))$.
    For the Boothby-Wang fibration, we have $\pi^{*}(\omega) = d\eta$. For vectors $v,w \in Ker(\eta_{p})$,
    \begin{flalign*}
        \Tilde{g}(v,\phi w) = d\eta(v,w) &= \pi^{*}\omega(v,w)\\
        &= \omega (d\pi v, d\pi w)\\
        &= g(d\pi v, J\left(d\pi w\right))\\
        &= \frac{1}{e^{2f}}\Tilde{g}(v, \widetilde{Jd\pi w}).
    \end{flalign*}
    Substituting $w=e_{1}$ and varying $v$ in the basis constructed above, we get
    \begin{equation}
        \begin{gathered}
            \Tilde{g}(\phi e_{1}, \widetilde{Jd\pi e_{1}}) = e^{2f},\\
            \Tilde{g}(v, \widetilde{Jd\pi e_1}) = 0 \text{ for all }v\in \{e_1, e_2, \dots, e_n, \phi e_2, \dots, \phi e_n\}.
        \end{gathered}
    \end{equation}
    In other words,
    \begin{equation*}
        \widetilde{Jd\pi e_{1}} = e^{2f}\phi e_1.
    \end{equation*}
    Since $\pi$ is a conformal submersion, the image set of a basis will again be a basis of the corresponding tangent space. Let $\{f_1, f_2,\dots, f_{2n}\}$ be the basis of $T_{\pi p}B$ which is the image of the basis constructed above, with elements in the order specified. With this notation, we have,
    \begin{equation*}
        J f_{1} = e^{2f} f_{n+1}.
    \end{equation*}
    Similarly, we can prove that,
    \begin{equation*}
        J f_{n+1} = -e^{2f} f_{1}.
    \end{equation*}
    Since $J$ is an almost complex structure, we have $J^{2} f_{1} = -f_{1}$. Using the identities above, we get $e^{4f} = 1$. In other words, $f=0$. Thus, for all $V,W \in \Gamma(Ker(\eta))$, we have $\tilde{g}(V,W) = g(d\pi V,d\pi W)$, which proves both the statements in the theorem.
\end{proof}
The construction of the basis of the form $\{e_{1}, e_{2}, \dots, e_{n}, \phi e_{1}, \phi e_{2}, \dots ,\phi e_{n}\}$ can be done locally, in a coordinate neighbourhood. Such a basis is called a local $\phi$-basis. For the remainder of the article, $\{f_{1},f_{2},\dots , f_{2n}\}$ will denote the (ordered) image of a local $\phi$-basis, under the derivative map $d\pi$.

As a special case of Theorem \ref{thm:confrig} when restricted to $(\kappa,\mu)$-manifolds, we obtain the following result.
\begin{corollary}\label{thm:confrigku}
    Suppose $\pi : M \rightarrow B$ is the projection map of a regular compact $(\kappa,\mu)$-manifold $(M,\eta,\Tilde{g}, \phi)$ onto a symplectic manifold $(B, \omega,g,J)$. If $\pi$ is a conformal submersion, then $M$ is Sasakian and $B$ is Kähler.
\end{corollary}
\begin{proof}
    From the previous lemma, we know that $\pi$ is a Riemannian submersion and $M$ is a $K$-contact manifold. Since every $(\kappa,\mu)$-manifold which is also $K$-contact is Sasakian, we have the first part of the result. By Hatakeyama's result (\cite{SKC}), $B$ is Kähler.
\end{proof}
Since the above result asserts that the projection map cannot be a conformal submersion in case of a non-Sasakian compact regular $(\kappa, \mu)$-manifold, we investigate the metrics with non-trivial deviation from being a conformal metric. To this end, we consider the bilinear form, corresponding to any associated metric $g$ on $(B, \omega)$ and $f\in C^{\infty}(M)$, given by
\begin{equation*}
    A(X,Y) = \tilde{g}(X,Y)-e^{2f}g(d\pi X, d\pi Y),
\end{equation*}
for all $X,Y \in \Gamma(Ker(\eta))$. There exists a unique (1,1)-tensor field $T$ for which the above equation can be reformulated as
\begin{equation}
    \Tilde{g}(TX, Y) = \Tilde{g}(X,Y)-e^{2f}g(d\pi X, d\pi Y),
\end{equation}
for all $X,Y \in \Gamma(Ker(\eta))$.

Henceforth, we will call $T$ the \textbf{error tensor}. We will try to solve for triples $(T,g,f)$ which satisfy \eqref{eq:main}. We begin by exploring the properties of the error tensor.
\begin{lemma}\label{lem:Tprop}
    The (1,1)-tensor field $T$ defined above has the following properties:
    \begin{enumerate}[i)]
        \item $T$ is symmetric with respect to $\Tilde{g}$,
        \item Every eigenvalue of $T$ is strictly less than 1, \label{lem:Tprop2}
        \item $\mathcal{L}_{\xi}T = 2\phi hT - 2\phi h-2(\xi f)(I-T)$.
    \end{enumerate}
\end{lemma}
\begin{proof}
    \begin{enumerate}[$i)$]
        \item Symmetry of $T$ follows from the symmetry of $g$ and $\Tilde{g}$. 
        \item Since $T$ is a symmetric tensor field, all of its eigenvalues are real. Suppose $v$ is a unit eigenvector of $T$ corresponding to an eigenvalue $\nu$. From equation \eqref{eq:main},
        \begin{flalign*}
            \nu = \Tilde{g}(Tv,v) &= \Tilde{g}(v,v) - e^{2f}g(d\pi(v), d\pi(v))\\
            &= 1 - e^{2f}\Vert d\pi(v) \Vert^{2}
        \end{flalign*}
        Thus,
        \begin{equation*}
            1-\nu = e^{2f}\Vert d\pi(v) \Vert^{2} > 0.
        \end{equation*}
        \item Taking the Lie derivative of equation \eqref{eq:main} along $\xi$, for horizontally lifted vector fields $\tilde{X},\tilde{Y}$, we get, using the properties mentioned in theorem \ref{thm:ghprop},
        \begin{flalign*}
            0 &= \mathcal{L}_{\xi}\tilde{g}(\tilde{X},\tilde{Y}) - 2(\xi f)e^{2f}g(X,Y) -\mathcal{L}_{\xi}g(T\tilde{X},\tilde{Y})\\
            &= 2\tilde{g}(h\tilde{X}, \phi \tilde{Y}) - 2(\xi f)\left( \tilde{g}(\tilde{X}-T\tilde{X},\tilde{Y}) \right)-2\tilde{g}(hT\tilde{X},\phi \tilde{Y}) - \tilde{g}((\mathcal{L}_{\xi}T)\tilde{X}, \tilde{Y})\\
            &= \tilde{g}\left((\phi hT-\phi h-(\xi f)(I-T))\tilde{X},\tilde{Y} \right)-\frac{1}{2}\tilde{g}((\mathcal{L}_{\xi}T)\tilde{X}, \tilde{Y}).
        \end{flalign*}
    Since the above identity holds pointwise for all vector fields $\tilde{X}$ and $\tilde{Y}$, we get
    \begin{equation*}
        \mathcal{L}_{\xi}T = 2\phi hT - 2\phi h-2(\xi f)(I-T).
    \end{equation*}
    \end{enumerate}
\end{proof}
Since the bi-Legendrian structure defined by the eigenspaces of $h$ play an important role in the study of $(\kappa,\mu)$-manifolds, we consider error tensors which preserve the bi-Legendrian structure. More precisely, we analyse the solutions to \eqref{eq:main} in which the error tensor maps the Legendrian distributions $\mathscr{D}_{h}(\lambda)$ and $\mathscr{D}_{h}(-\lambda)$ to themselves. We begin by stating the following fact which can be verified easily.
\begin{lemma}\label{lem:commeq1}
    Let $T$ be the error tensor as in \eqref{eq:main}. Then,
    the tensor field $h$ commutes with $T$ if and only if $T(\mathscr{D}_{h}(\lambda)) \subseteq \mathscr{D}_{h}(\lambda)$ and $T(\mathscr{D}_{h}(-\lambda)) \subseteq \mathscr{D}_{h}(-\lambda)$.
\end{lemma}
We have already seen that the metric $g$ on the base space cannot be in the conformal class of $\tilde{g}$. However, the condition $hT=Th$ ensures that the metric $g$ respects the orthogonality of vectors coming from $\mathscr{D}_{h}(\lambda)$ and $\mathscr{D}_{h}(-\lambda)$ in the following sense.
\begin{lemma}\label{lem:commeq2}
    Let $(T,f,g)$ be a solution to equation \eqref{eq:main}. Then the following are equivalent:
    \begin{enumerate}
        \item $hT=Th$.
        \item $g(d\pi X,d\pi Y) = 0$ whenever $X\in \mathscr{D}_{h}(\lambda), Y\in\mathscr{D}_{h}(-\lambda)$ . 
    \end{enumerate}
\end{lemma}
\begin{proof}
    Equation \eqref{eq:main} can be rearranged to obtain
    \begin{equation}\label{eq:main2}
        g(d\pi X, d\pi Y) = \frac{\tilde{g}(X-TX,Y)}{e^{2f}},
    \end{equation}
    for all $X,Y \in \Gamma(Ker(\eta)).$
    Suppose $hT=Th$. Let $X\in\mathscr{D}_{h}(\lambda)$ and $Y\in \mathscr{D}_{h}(-\lambda)$. From Lemma \ref{lem:commeq1}, $T$ preserves the eigenspaces of $h$. Thus, $TX\in \mathscr{D}_{h}(\lambda)$ and consequently, $X-TX$ is $\tilde{g}$-orthogonal to $Y$. Using equation \eqref{eq:main2}, we obtain $g(d\pi X,d\pi Y) = 0$.

    For the converse, suppose statement 2. in the statement of Lemma \ref{lem:commeq2} holds. For $X\in\mathscr{D}_{h}(\lambda)$ and $Y\in \mathscr{D}_{h}(-\lambda)$, we have, by equation \eqref{eq:main2}, $\tilde{g}(TX,Y)=0$. In other words, $TX\in \left( \mathscr{D}_{h}(-\lambda) \right)^{\perp} = \mathscr{D}_{h}(\lambda)$. A similar argument can be used to show that $TX\in\mathscr{D}_{h}(-\lambda)$ $X\in\mathscr{D}_{h}(-\lambda)$ and thus $T$ preserves the eigenspaces of $h$. By Lemma \ref{lem:commeq1}, this is equivalent to $hT=Th$.
\end{proof}
Thus, we consider solutions to equation \eqref{eq:main} for which $hT=Th$. Since $h$ and $T$ are commuting diagonalizable operators, they are simultaneously diagonalizable and share a common eigenbasis. The following result shows that there is a common eigenbasis which is a (local) $\phi$-basis.
\begin{lemma}\label{lem:comphibasis}
    Let $T$ be the error tensor for the Boothby-Wang fibration of a closed regular non-Sasakian $(\kappa,\mu)$-manifold. If $hT=Th$, then the Boeckx Index $I_{M}$ must satisfy $\vert I_{M} \vert > 1$. For $\vert I_{M} \vert > 1$, the tensors $h$ and $T$ have a common orthonormal eigenbasis at every point, of the form $\{e_1, e_2, \dots , e_n, \phi e_1, \phi e_2, \dots , \phi e_n\}$, where $e_i \in \mathscr{D}(\lambda)$ for all $i \in \{1,2,\dots , n\}$. 
\end{lemma}
\begin{proof}
    Since $T\vert_{\mathscr{D}(\lambda)}$ is a symmetric operator, it has an eigenbasis $\{e_{1}, e_{2}, \dots ,e_{n} \}$ of orthonormal vectors. For every $i\in \{1,2,\dots, n\}$, let $\lambda_{i}$ be the eigenvalue corresponding to the eigenvector $e_{i}$. We claim that $\{e_{1}, e_{2}, \dots , e_{n}, \phi e_{1}, \phi e_{2}, \dots , \phi e_{n}\}$ is the required basis.
    In order to prove this claim, it is sufficient to show that $\phi e_{i}$ is an eigenvector of $T$ for all $i\in\{1,2,\dots, n\}$. To see this, we consider the equation $hT = Th$ and take its Lie derivative along $\xi$ to get
    \begin{flalign*}
        0 =& \mathcal{L}_{\xi}(Th)-\mathcal{L}_{\xi}(hT)\\
        =&   (\mathcal{L}_{\xi}T)h + T(\mathcal{L}_{\xi}h) -(\mathcal{L}_{\xi}h)T - h(\mathcal{L}_{\xi}T)\\
        =&
        (2\phi hT - 2\phi h-2(\xi f)(I-T))h  + T((2-\mu)\phi h+2(1-\kappa)\phi)\\
        & 
        -((2-\mu)\phi h+2(1-\kappa)\phi )T  -h(2\phi hT - 2\phi h-2(\xi f)(I-T)).
    \end{flalign*}
    Rearranging the above equation and evaluating it at $e_{i}$ using $he_{i} = \lambda e_{i}$ and $Te_{i}=\lambda_{i}e_{i}$, we get,
    \begin{equation}\label{eq:eigT}
        (I_{M}+1)T(\phi e_{i})=(\lambda_{i}(I_{M}-1)+2)
    \end{equation}
    For $I_{M}=-1$, we get $\lambda_{i}=1$, which is a contradiction to the fact that the eigenvalues of $T$ are strictly lesser than 1.
    For all other values of $I_{M}$, equation \eqref{eq:eigT} can be rewritten as
    \begin{equation}\label{eq:eigT2}
        T(\phi e_{i})=\left(\frac{\lambda_{i}(I_{M}-1)+2}{1+I_{M}}\right)\phi e_{i}.
    \end{equation}
    Thus, $\phi e_{i}$ is an eigenvector of $T$, as required. Let $\lambda_{n+i}$ denote the eigenvalue corresponding to the eigenvector $\phi e_{i}$ of $T$. By equation \eqref{eq:eigT2}, we have
    \begin{equation}\label{eq:line}
        (I_{M}-1)\lambda_{i}-(1+I_{M})\lambda_{n+i}+2=0.
    \end{equation}
    Equation \eqref{eq:line} is an equation of a straight line in the variables $\lambda_{i}$ and $\lambda_{n+i}$. The line has slope $\left(\frac{1+I_{M}}{I_{M}-1}\right)$ and passes through the point $(1,1)$. Since we need the eigenvalues to be strictly lesser than $1$, we need the slope of the above line to be strictly positive. Equivalently, $\vert I_{M} \vert > 1.$
\end{proof}
As a part of the proof of Lemma \ref{lem:comphibasis}, we have also proved the following result, which we state separately for clarity.
\begin{corollary}\label{lem:line}
    Let $\{e_{1}, e_{2}, \dots ,e_{n}, \phi e_{1},\phi e_{2}, \dots , \phi e_{n}\}$ be the $\phi$-basis constructed in Lemma \ref{lem:comphibasis}. For every $i\in \{1,2,\dots ,n\}$, let $\lambda_{i}$ be the eigenvalue of $T$ corresponding to $e_{i}$ and let $\lambda_{n+i}$ be the eigenvalue of $T$ corresponding to $\phi e_{i}$. Then,
    \begin{equation}
        (I_{M}-1)\lambda_{i}-(1+I_{M})\lambda_{n+i}+2=0.
    \end{equation}
\end{corollary}
Now, we prove a result which shows that the eigenvalues $\left\{ \lambda_{i} \right\}_{i=1}^{2n}$ mentioned above satisfy another set of equations. This allows us to determine the error tensor $T$ in terms of the function $f$ in equation \eqref{eq:main}.
\begin{lemma}\label{lem:hyp}
    Let $\lambda_{i}$ be the eigenvalues of $T$ mentioned in Corollary \ref{lem:line}. If $(T,f,g)$ is a solution to equation \eqref{eq:main}, then
    \begin{equation}\label{eq:hyp}
        (1-\lambda_{i})(1-\lambda_{n+i})=e^{4f}
    \end{equation}
\end{lemma}
\begin{proof}
    Since $\pi$ is the projection map of the Boothby-Wang fibration, we have $\pi^{*}\omega = d\eta$. For vectors $v,w \in Ker(\eta)$,
    \begin{flalign}
        \Tilde{g}(v,\phi w) = d\eta(v,w) &= \pi^{*}\omega(v,w)\nonumber \\
        &= \omega (d\pi (v), d\pi (w))\nonumber \\
        &= g(d\pi (v), J\left(d\pi (w)\right))\nonumber \\
        &= \frac{1}{e^{2f}}\left( \Tilde{g}(v-Tv, \widetilde{J d\pi (w)})\right).\label{eq:eq2}
    \end{flalign}
    Consider the basis constructed in the proof of Lemma \ref{lem:comphibasis}. Substituting $w=e_1$ and varying $v$ in the elements of the chosen basis, we get,
    \begin{equation*}
        \Tilde{g}(v, \widetilde{Jd\pi(e_1)}) = 0 \text{ for all }v\in \{e_1, e_2, \dots, e_n, \phi e_2, \dots, \phi e_n\}.
    \end{equation*}
    For $v=\phi e_{1}$, we have
    \begin{flalign*}
        1=\Tilde{g}(\phi e_{1},\phi e_{1})&=\frac{1}{e^{2f}}\left(\Tilde{g}(\phi e_{1}-T\phi e_{1},\widetilde{Jd\pi (e_{1})} \right)\\
        &= \left(\frac{1-\lambda_{n+1}}{e^{2f}}\right)\Tilde{g}(\phi e_{1},\widetilde{Jd\pi (e_{1})})
    \end{flalign*}
    Thus, $\widetilde{Jd\pi (e_1)} = \left(\frac{e^{2f}}{1-\lambda_{n+1}}\right)\phi e_{1}$. Let $\{f_1, f_2, \dots, f_{2n}\}$ be the image of the chosen basis under $d\pi$, with elements in the same order. We have, $\widetilde{Jf_1}=\left(\frac{e^{2f}}{1-\lambda_{n+1}}\right) \phi e_{1}$. Applying the derivative map of $\pi$, we get $Jf_1 = \left(\frac{e^{2f}}{1-\lambda_{n+1}}\right)f_{n+1}$. By varying both $v$ and $w$ in equation \eqref{eq:eq2} from the elements of the chosen basis and performing similar computations, we can conclude that
    \begin{equation}\label{eq:J1}
        Jf_{i} = \left(\frac{e^{2f}}{1-\lambda_{n+i}}\right)f_{n+i},\qquad Jf_{n+i} = -\left(\frac{e^{2f}}{1-\lambda_{i}}\right)f_{i},
    \end{equation}
    for all $i\in\{1,2, \dots, n\}$. Since $J$ is an almost complex structure, we have, for all $i\in\{1,2,\dots, n\}$,
    \begin{flalign*}
        -f_{i} =& J^{2}f_i \\
        =&J\left( \left(\frac{e^{2f}}{1-\lambda_{n+i}}\right)f_{n+i} \right)\\
        =&\left( \frac{e^{2f}}{1-\lambda_{n+i}} \right)\left( -\frac{e^{2f}}{1-\lambda_{i}}f_{i} \right)
    \end{flalign*}
    Therefore, we can conclude that the eigenvalues of $T$ in this case satisfy the equation
    \begin{equation}
        (1-\lambda_{i})(1-\lambda_{n+i})=e^{4f}.
    \end{equation}
\end{proof}
The intersection of the solution sets to equations \eqref{eq:line} and \eqref{eq:hyp} consists of two points. But due to the condition given in Lemma \ref{lem:Tprop} \ref{lem:Tprop2}, only one of the solutions is permissible. See Figure \ref{fig:eigsol} for a representative depiction.
\begin{figure}[!htbp]
    \centering
        \begin{tikzpicture}[scale=0.5]
            \begin{axis}[
                x=1cm, y=1cm,
                axis lines=middle,
                axis line style={<->},
                ymin=-7, ymax=9,
                xmin=-9, xmax=11,
                yticklabel=\empty,
                xticklabel=\empty,
                xtick style = {draw=none},
                ytick style = {draw=none},
                xlabel = {$\lambda_{i}$},
                ylabel = {$\lambda_{n+i}$},
                x label style={anchor=west},
                y label style={anchor=south},
                ]
                \addplot [
                    domain=-10:0.9,
                    samples=200
                ] {1+2/((x)-1)};
                \addplot[
                    domain=1.1:12,
                    samples=200,
                ] {1+2/((x)-1)};
                \addplot[
                    domain=-9:11,
                    dashed,
                    ]{1};
                \addplot [
                    mark=none,
                    dashed,
                    ] coordinates {(1, \pgfkeysvalueof{/pgfplots/ymin}) (1, \pgfkeysvalueof{/pgfplots/ymax})};
                \addplot[
                    domain=-12:12,
                    ]{(x+2)/3};
            \end{axis}
        \end{tikzpicture}
    \caption{The graphs of curves determining the eigenvalues of $T$ for $\vert I_{M} \vert > 1$.}
    \label{fig:eigsol}
\end{figure}
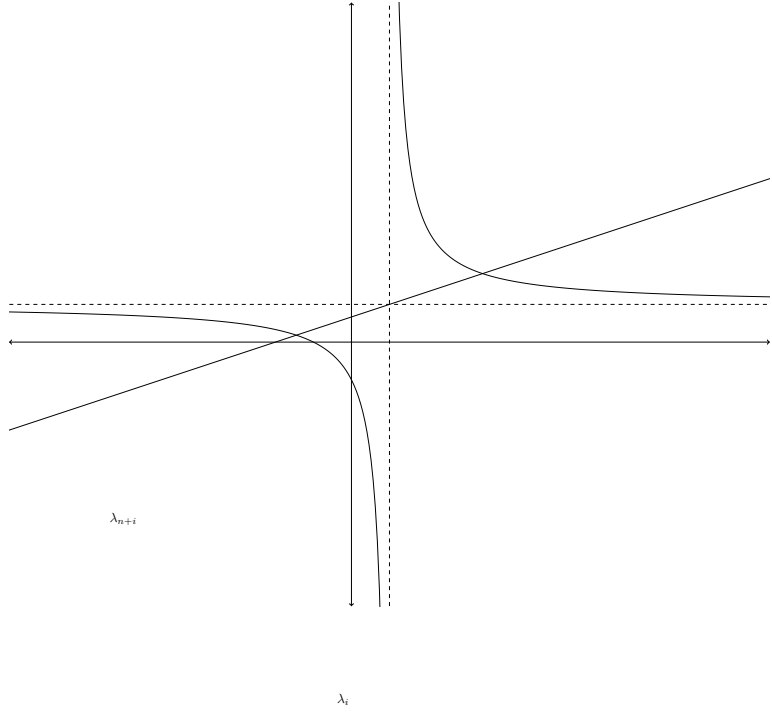
Thus, the above results determine the error tensors which arise as solutions to equation \eqref{eq:main} and commute with the tensor field $h$. More precisely, we have the following result.
\begin{lemma}\label{lem:T}
    Suppose $(T,f,g)$ is a solution to equation \eqref{eq:main} and $hT=Th$. Then,
    \begin{equation}\label{eq:T}
        T\vert_{\mathscr{D}(\lambda)}=\left(\frac{ 1-\sqrt{\frac{I_M+1}{I_M-1}}e^{2f}}{\lambda}\right) h\vert_{\mathscr{D}(\lambda)}\qquad T\vert_{\mathscr{D}(-\lambda)}=\left(\frac{1-\sqrt{\frac{I_M-1}{I_M+1}}e^{2f}}{\lambda}\right)h\vert_{\mathscr{D}(-\lambda)} 
    \end{equation}
\end{lemma}
\begin{proof}
    Under the hypothesis of the statement, we know that the eigenvalues of the error tensor satisfy equations \eqref{eq:line} and \eqref{eq:hyp}. Solving the equations simultaneously gives two solutions for each eigenvalue of $T$. Out of the two solutions, only one of them is lesser than 1. Thus, there is only one acceptable solution. A straightforward computation shows that the acceptable solution set for the eigenvalues $\left\{\lambda_{i} \right\}_{i=1}^{2n}$ of $T$ are given by
    \begin{equation*}
        \lambda_{i}= 1-\sqrt{\frac{I_M+1}{I_M-1}}e^{2f}\text{ and }\lambda_{n+i}= 1-\sqrt{\frac{I_M-1}{I_M+1}}e^{2f}.
    \end{equation*}
    Thus, the error tensor has two eigenspaces of dimension $n$ each, which coincide with the eigenspaces of the tensor field $h$. Therefore, we have the desired expressions for $T$.
\end{proof}
The function $f$ in equation \eqref{eq:main} determines $T$ completely. We now prove a result which establishes rigidity of the metric $g$ appearing in the solutions to equation \eqref{eq:main}.
\begin{lemma}\label{lem:uniqmet}
    Suppose $(T,f,g)$ is a solution to equation \eqref{eq:main} and $hT=Th$. Then,
    \begin{equation}\label{eq:uniqmet}
        g(X,Y) = \frac{\lvert I_{M} \rvert}{\sqrt{I_{M}^{2}-1}}\left(\tilde{g}(\tilde{X},\tilde{Y})+\tilde{g}\left(\tilde{X}, \left(\frac{h}{1-\tfrac{\mu}{2}}\right)\tilde{Y}\right) \right),
    \end{equation}
    for all vector fields $X,Y$ on $B$.
\end{lemma}
\begin{proof}
    In the proof of Lemma \ref{lem:hyp}, we have obtained a basis $\{f_{i}\}_{i=1}^{2n}$ and an expression for the almost complex structure $J$ in terms of this basis (equation \eqref{eq:J1}). Combining this with the values of the eigenvalues of $T$ obtained in the proof of Lemma \ref{lem:T}, we get, for $i \in \{1,2,\dots, n\}$,
    \begin{equation*}
        Jf_{i}= \sqrt{\frac{I_M+1}{I_M-1}}f_{n+i}\text{ and }Jf_{n+i}=\sqrt{\frac{I_M-1}{I_M+1}}f_{i}.
    \end{equation*}
    Since the expression for $J$ in the above basis is independent of $f$ and $T$, we conclude that any solution to equation \eqref{eq:main} comprises a unique metric $g$. In order to obtain an expression for the unique metric, we may substitute $f\equiv 0$ in equation \eqref{eq:main} and simplify in order to obtain the desired result.
\end{proof}

The implications of the imposition of the condition $hT=Th$ are summarized in Theorem \ref{thm:main} and proved above (Lemma \ref{lem:comphibasis}, Lemma \ref{lem:T} and Lemma \ref{lem:uniqmet}).
\begin{remark}\label{rem:cansas1}
    The metric $g$ described in the third part of Theorem \ref{thm:main} can be checked to be a Kähler metric. Naturally, it corresponds to a Sasakian metric on $M$, by \cite{SKC}. This Sasakian metric has also been studied, in a different context, in \cite{SPS} where the authors describe it as the canonical Sasakian structure associated to the $(\kappa , \mu)$-structure on $M$.
\end{remark}
\begin{remark}\label{rem:cansas2}
    In \cite{CFKU}, the authors describe the canonical structure of the base space of a simply connected, complete $(\kappa,\mu)$-manifold. For $\vert I_{M} \vert >1$, the base space admits the structure of the complexification of a sphere. The structure tensors of the space obtained by the authors coincide with tensors obtained as the unique solution obtained here. 
\end{remark}
\noindent Theorem \ref{thm:main} can be combined with Lemma \ref{lem:commeq2} to obtain the following result.
\begin{corollary}\label{cor:rigbase}
    There is a unique Riemannian metric $g$ on the base space of a Boothby-Wang fibration of a compact regular $(\kappa,\mu)$-manifold which satisfies the condition
    \begin{equation}\label{eq:basecon}
        g(d\pi X,d\pi Y)= 0 \text{ whenever }X\in\mathscr{D}_{h}(\lambda), Y\in\mathscr{D}_{h}(-\lambda).
    \end{equation}
    Furthermore, $g$ defines a Kähler structure on the base space.
\end{corollary}
The above corollary, along with the correspondence between Kähler and Sasakian manifolds can be used to obtain Theorem \ref{thm:rig}.
\begin{proof}[Proof of Theorem \ref{thm:rig}]
    Suppose $(M,\eta , \Bar{g},\Bar{\phi})$ is a $K$-contact structure as specified in the above statement. The metric $\Bar{g}$ induces a Riemannian metric $g$ on the base space of the Boothby-Wang fibration of $M$ (See Section \ref{subsec:CMM} for details). In this case, the corresponding projection map $\pi$ is a Riemannian submersion. Thus, the Riemannian metric $g$ satisfies the condition specified in equation \eqref{eq:basecon}. Thus, we conclude that $g$ is the unique Kähler metric obtained as a solution to equation \eqref{eq:main}. This forces $\Bar{g}= \pi^{*}g$. The second statement follows from remark \ref{rem:cansas1}.
\end{proof}
\section{Semi-Riemannian metrics on the base space}\label{sec:semi-riem}
We have seen that there are no associated metrics on the base space (under the Boothby-Wang fibration) of a regular closed $(\kappa,\mu)$-manifold with Boeckx index $\vert I_{M} \vert > 1$ for which the associated error tensor preserves the Legendrian distributions given by the eigenspaces of $h$. So we look at the possibilities of semi-Riemannian metrics on the base space associated with the symplectic structure and corresponding error tensors. As before, we start out with the assumption that $\pi:M\rightarrow B$ is a semi-Riemannian submersion. Suppose the metric on $B$ has vectors of negative length. Then, we get vectors with negative length on the total space, as well, which is not possible. Thus, $B$ is forced to be a Riemannian manifold and we land in the previous case. In light of this observation, we define the error tensor, as before, by the equation,
\begin{equation}
    \Tilde{g}(T\tilde{X}, \tilde{Y}) = \Tilde{g}(\tilde{X},\tilde{Y})-g(X, Y),
\end{equation}
where $X,Y\in \Gamma(B)$. Since the analysis follows a path similar to the one in the previous section, we will skip the proofs of the statements which follow directly from similar computations. Symmetry of the tensor field $T$ follows immediately. A formula for the Lie derivative of $T$ can also be obtained by carrying out computations similar to the ones in the proof of Lemma \ref{lem:Tprop}. However, the eigenvalues of the error tensor being bounded above by $1$ in the case of Riemannian metrics was a consequence of positive definiteness, which is dropped here. Instead, we have the following result.
\begin{lemma}\label{lem:eigTS}
    The error tensor $T$ defined in equation \eqref{eq:main3} can not have an eigenvalue equal to 1.
\end{lemma}
\begin{proof}
    Suppose $\tilde{X}$ is an eigenvector of $T$ corresponding to the eigenvalue $1$. Then, by equation \eqref{eq:main3}, we get $\tilde{g}(\tilde{X},\tilde{Y}) = \tilde{g}(\tilde{X},\tilde{Y}) - g(X,Y)$, for all vectors $\tilde{Y}\in \Gamma(Ker(\eta))$. Thus, we have $g(X,Y)=0$, for all vectors $Y\in \Gamma(B)$, contradicting the non-degeneracy of the tensor $g$. Therefore, $T$ can not have an eigenvalue equal to $1$.
\end{proof}
Following an analysis similar to the one in the case of Riemannian metrics on the base space, we can conclude that the error tensor $T$ commutes with the tensor field $h$ if and only if the condition specified in \eqref{eq:basecon} holds. Hence, we investigate solutions $(T,g)$ to equation \eqref{eq:main3} for which $hT=Th$. Since the proof of Lemma \ref{lem:comphibasis} does not depend on the positive-definiteness of the metric tensor on the base, we can conclude that $h$ and $T$ share a common eigenbasis which is a local $\phi$-basis and the eigenvalues $\{\lambda_{i}\}_{1}^{2n}$ of $T$ satisfy the equation 
\begin{equation}
        (I_{M}-1)\lambda_{i}-(1+I_{M})\lambda_{n+i}+2=0.
\end{equation}
The equation of the hyperbola given by \eqref{eq:hyp} was derived using the fact that $J$ is an almost complex structure on the base space. However, this is not the case here since the base space is equipped with a $(1,1)$-tensor field $F$ which satisfies $F^{2}=1$. The proof of Lemma \ref{lem:hyp} can be modified to obtain the following result.
\begin{lemma}
    Let $(T,g)$ be a solution to equation \eqref{eq:main3}. Then, the eigenvalues $\{\lambda_{i}\}_{1}^{2n}$ satisfy the relations
    \begin{equation}\label{eq:hyp2}
        (1-\lambda_{i})(1-\lambda_{n+i})=-1,
    \end{equation}
    for all $i\in\{1,2,\dots ,n\}$.
\end{lemma}
The line has slope equal to $\left( \frac{1+I_{M}}{I_{M}-1}\right)$, and passes through the point $(1,1)$. The line and the hyperbola have non-empty intersection only if the slope of the line is negative (see Figure \ref{fig:eigsolsemi} for a particular case). This is equivalent to the condition $\vert I_{M} \vert < 1$. Under this assumption on the index, the points of intersections of the line \eqref{eq:line} and hyperbola \eqref{eq:hyp2} are given by
\begin{equation}
    p_{1} \coloneqq \left( 1+\sqrt{\frac{1+I_{M}}{1-I_{M}}},1-\sqrt{\frac{1-I_{M}}{1+I_{M}}}\right)\text{ and }p_{2}\coloneqq \left( 1-\sqrt{\frac{1+I_{M}}{1-I_{M}}},1+\sqrt{\frac{1-I_{M}}{1+I_{M}}}\right).
\end{equation}
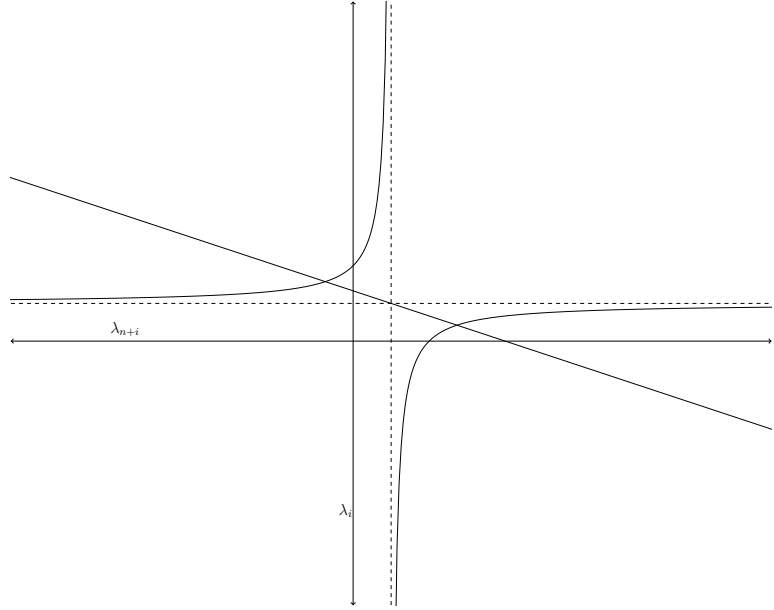
\begin{figure}[!htbp]
    \centering
        \begin{tikzpicture}[scale=0.5]
            \begin{axis}[
                x=1cm, y=1cm,
                axis lines=middle,
                axis line style={<->},
                ymin=-7, ymax=9,
                xmin=-9, xmax=11,
                yticklabel=\empty,
                xticklabel=\empty,
                xtick style = {draw=none},
                ytick style = {draw=none},
                xlabel = {$\lambda_{i}$},
                ylabel = {$\lambda_{n+i}$},
                x label style={anchor=west},
                y label style={anchor=south},
                label style = {font = \large}
                ]
                \addplot [
                    domain=-10:0.9,
                    samples=200
                ] {1+1/(1-x)};
                \addplot[
                    domain=1.1:12,
                    samples=200,
                ] {1+1/(1-x)};
                \addplot[
                    domain=-9:11,
                    dashed,
                    ]{1};
                \addplot [
                    mark=none,
                    dashed,
                    ] coordinates {(1, \pgfkeysvalueof{/pgfplots/ymin}) (1, \pgfkeysvalueof{/pgfplots/ymax})};
                \addplot[
                    domain=-12:12,
                    ]{(-x+4)/3};
            \end{axis}
        \end{tikzpicture}
    \caption{The graphs of curves determining the eigenvalues of $T$}
    \label{fig:eigsolsemi}
\end{figure}

Thus, for every $i\in\{1,2,\dots,n\}$, the ordered pair $(\lambda_{i},\lambda_{n+i})$ is either $p_{1}$ or $p_{2}$. Therefore, there are $2^{n}$ solutions for the (ordered) set of eigenvalues $\{\lambda_{i}\}_{i=1}^{2n}$ of $T$. In other words, there are $2^{n}$ many $(1,1)$-tensor fields that can appear in the solutions to the equation \ref{eq:main3}. Since these tensor fields are defined using a $\phi$-basis which is defined in a neighbourhood, each of the solutions is a smooth $(1,1)$-tensor field. Thus, we obtain $2^{n}$ semi-Riemannian metrics on $B$ given by the equation $g(X,Y) = \tilde{g}(\tilde{X}-T\tilde{X},\tilde{Y})$. These semi-Riemannian metrics define almost para-Kähler structures on the base space. These tensors can described locally, in terms of the basis $\{f_{i}\}_{i=1}^{2n}$, which is the image of the chosen $\phi$-basis, as follows. Let $S$ be a subset of $\{1,2,\dots , n\}$. Let $a_{0}=\sqrt{\frac{1+I_{M}}{1-I_{M}}}$. Define a $(1,1)$-tensor field $F_{S}$ on $B$ by the equations
\begin{flalign*}
    F_{S}f_{i}&=\begin{cases}
                    a_{0}f_{n+i}\:\:\text{, if }i\in S\\
                    -a_{0}f_{n+i}\:\:\text{, if }i\in \{1,2,\dots ,n\}\setminus S
                \end{cases}\\
    F_{S}f_{n+i}&=\begin{cases}
                    \frac{1}{a_{0}}f_{i}\:\:\text{, if }i\in S\\
                    -\frac{1}{a_{0}}f_{i}\:\:\text{, if }i\in \{1,2,\dots ,n\}\setminus S
                \end{cases}
\end{flalign*}
Also, define $(0,2)$-tensors $g_{S}$ by the equations
\begin{flalign*}
    g_{S}(f_{i},f_{i})&=\begin{cases}
                            -a_{0}\:\:\text{, if }i\in S\\
                            a_{0}\:\:\text{, if }i\in \{1,2,\dots ,n\}\setminus S
                        \end{cases}\\
    g_{S}(f_{n+i},f_{n+i})&=\begin{cases}
                                \frac{1}{a_{0}}\:\:\text{, if }i \in S\\
                                -\frac{1}{a_{0}}\:\:\text{, if }i\in \{1,2,\dots ,n\}\setminus S
                            \end{cases}\\
    g_{S}(f_{i},f_{j})&=0\:\:\text{, if }i\neq j
\end{flalign*}
One can check that $(B,\omega, g_{S},F_{S})$ is an almost para-Kähler structure for every $S$. The semi-Riemannian metrics corresponding to these almost para-Kähler structures are precisely the ones that arise as solutions to \eqref{eq:main3}.
The above analysis proves Theorem \ref{thm:main2}.
\begin{remark}
        \item If $(\lambda_{i},\lambda_{n+i})$ is chosen to be $p_{1}$ for every $i$, the solution corresponds to the choice $S=\{1,2,\dots, n\}$ in the above description. The para-contact structure corresponding to this almost para-Kähler structure was studied in \cite{SPS} and described as the canonical para-Sasakian structure associated with a non-Sasakian $(\kappa,\mu)$-manifold with index $\vert I_{M} \vert < 1$. Thus, the corresponding semi-Riemannian metric on the base space defines a para-Kähler structure.
\end{remark}
\bibliography{References}% common bib file
\section*{Statements and Declarations}
\begin{itemize}
    \item \textbf{Funding:} Author Sannidhi Alape was supported by CSIR grant \text{09/1020(0143)/2019-EMR-I, DST}, Government of India. Authors Atreyee Bhattacharya and Dheeraj Kulkarni declare that no funds, grants, or other support were received by them for the preparation of this manuscript.
    \item \textbf{Competing interests:} The authors have no relevant financial or non-financial interests to disclose.
    \item \textbf{Data availablility statement:} No data was used to derive the results in this article.
    \item \textbf{Authors' contributions:} All authors contributed to the study conception and design. The mathematical ideas contained in this article were contributed by Sannidhi Alape, Atreyee Bhattacharya and Dheeraj Kulkarni. The first draft of the manuscript was prepared jointly by all authors. All authors commented on previous versions of the manuscript. All authors read and approved the final manuscript.
\end{itemize}

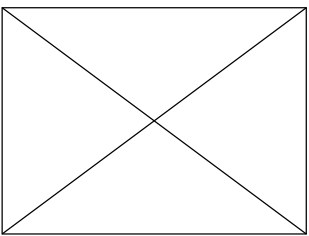
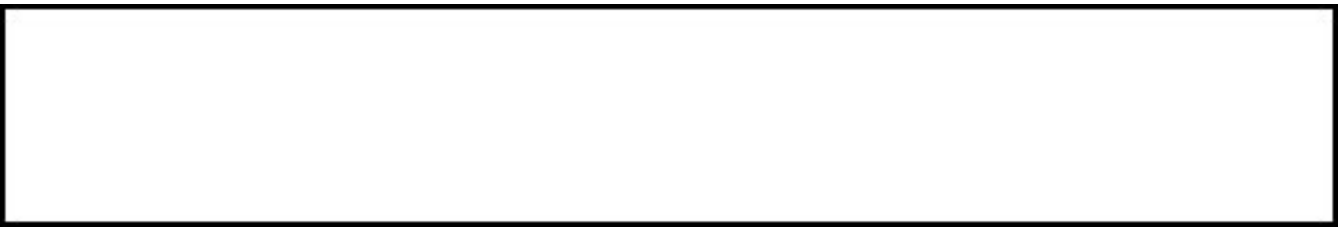
\end{document}